\documentclass[a4paper,11pt,reqno]{amsart}
\usepackage{graphicx}
\usepackage{a4wide}
\usepackage{amssymb}
\usepackage{amsthm}
\usepackage{eucal}
\usepackage[pdftex,colorlinks]{hyperref}

\numberwithin{equation}{section}

\newcommand{\R}{\mathbb R}

\newcommand{\N}{\mathbb N}

\newcommand{\be}{\begin{equation}}
\newcommand{\ee}{\end{equation}}
\newcommand{\ba}{\begin{eqnarray}}
\newcommand{\ea}{\end{eqnarray}}

\def\rank{\hbox{\rm rank$\,$}}

\def\ds{\displaystyle}

\def\h{\widehat}
\def\wt{\widetilde}

\newtheorem{theorem}{Theorem}[section]
\newtheorem{proposition}[theorem]{Proposition}
\newtheorem{remark}[theorem]{Remark}
\newtheorem{lemma}[theorem]{Lemma}
\newtheorem{corollary}[theorem]{Corollary}
\newtheorem{definition}[theorem]{Definition}

\newtheorem{assumption}[theorem]{Assumption}

\begin{document}


\title{Controllability of evolution equations with memory}

\author[F. W. Chaves Silva]{Felipe Wallison Chaves-Silva}
\address{Departament of Mathematics, Federal University of Pernambuco, CEP: 50740-540, Recife, PE, Brazil}
\email{fchaves@dmat.ufpe.br}

\author[X. Zhang]{Xu Zhang}
\address{School of Mathematics, Sichuan University, Chengdu 610064, Sichuan Province, China.}
\email{zhang$\_$xu@scu.edu.cn}

\author[E. Zuazua]{Enrique  Zuazua}
\address{[1] DeustoTech, Fundaci\'on Deusto, Avda Universidades, 24, 48007, Bilbao - Basque Country - Spain\\
	[2] Departamento de Matem\'aticas, Universidad Aut\'onoma de Madrid, Cantoblanco, 28049 Madrid - Spain\\
	[3] Facultad de Ingenier\'{\i}a, Universidad de Deusto, Avda. Universidades, 24, 48007, Bilbao - Basque Country - Spain.}
\email{enrique.zuazua@uam.es}

\begin{abstract}
This article is devoted to studying the null controllability of evolution equations with memory terms. The problem is challenging not only because the state equation contains memory terms but also because the classical controllability requirement at the final time has to be reinforced, involving the contribution of the memory term,  to ensure that the solution reaches the equilibrium.   Using duality arguments, the problem is reduced to the obtention of suitable observability estimates for the adjoint system. We first consider finite-dimensional dynamical systems involving memory terms and derive rank conditions for controllability. Then the null controllability property is established for  some parabolic  equations with memory terms, by means of Carleman estimates.
\end{abstract}



\maketitle

\section{Introduction}

The problem of controllability for evolution equations is a classical one. Starting from finite dimensional linear systems
(see \cite{15}), where controllability can be characterized by algebraic rank equations on the matrices generating the dynamics and taking account of the control action,  the theory has been adapted and extended  to more general systems including
infinite dimensional systems, and its nonlinear and stochastic counterparts (see e.g. \cite{AI, Coron, FI, Lions, Russell, Z2, Zua2} and the rich references therein).

However, most of the existing works are concerned with  evolution equations  involving memory terms
that are relevant from a physical point of view.
For instance, in \cite{GP} a modified Fourier's law was  introduced
 to correct the unphysical property
of instantaneous propagation for the heat equation   (e.g. \cite{Ca}), which results in a heat equation with
memory:
 \begin{equation}\label{1.2}
 \left\{
 \begin{array}{ll}\ds
 y_t-\sum_{i,j=1}^n  \left\{a^{ij}(x)\left[ay_{x_i}+\int_0^t b(t-s,x)y_{x_i}(s,x)ds\right]\right\}_{x_j}=u\chi_\omega(x) & \mbox{ in }
 Q,\\[2mm]
 \ds
 y= 0& \mbox{ on }\Sigma,\\[2mm]
\ds y(0)=y_0& \mbox{ in } \Omega.
 \end{array}
 \right.
 \end{equation}
Here $b(\cdot,\cdot)$ is a smooth memory kernel, $a\in \{0,1\}$ is a parameter, $\Omega \ (\subset \mathbb{R}^n, \ n\in\mathbb{N})$ is a bounded domain with a   $C^\infty$-smooth boundary $\partial \Omega$,  $T > 0$ is a given finite time horizon and  $\omega$ is a non-empty open subset of $\Omega$ where the control is applied. We also use the notation $Q = (0,T)\times\Omega $ and  $\Sigma = (0,T)\times \partial \Omega$, and denote by $\chi_\omega$ the characteristic function of $\omega$ and by $\nu=\nu(x)$ the unit outward
normal vector of $\Omega$ at $x\in \partial \Omega$. $x=(x_1,\cdots,x_n)^\top$ and $(a^{ij}(x)\big)_{n\times n}$ is a
given uniformly positive definite matrix with suitable smoothness.

The well-posedness and the propagation speed of these models
was analysed in \cite{YZ1}.

In the absence of memory term (i.e., $b(\cdot) \equiv 0$) and when $a=1$, this system becomes the following classical heat equation
  \begin{equation}\label{1.1}
 \left\{
 \begin{array}{ll}
 \displaystyle y_t-\sum_{i,j=1}^n \Big[a^{ij}(x)y_{x_i}\Big]_{x_j}=u\chi_\omega(x) & \hbox{ in }  Q, \\[2mm]
  y=0              &      \hbox{ on } \Sigma, \\[2mm]
  y(0)=y_0          & \hbox{ in }  \Omega,
 \end{array}
 \right.
 \end{equation}
 and  its null controllability properties  are by now well known.
For instance, it is
well-known (e.g. \cite{FI}) that for any given $T>0$ and
non-empty open subset $\omega$ of $\Omega$, the equation (\ref{1.1}) is
null controllable in $L^2(\Omega)$,
i.e., for any given $y_0\in L^2(\Omega)$, one can find a control $u\in
L^2((0,T)\times\omega)$ such that
the weak solution $y(\cdot)\in C([0,T];L^2(\Omega))\cap
C((0,T];H_0^1(\Omega))$ to (\ref{1.1}) satisfies
\begin{equation}\label{null}
y(T)=0.
\end{equation}

In this parabolic setting it is notable that, thanks to the  infinite
speed of propagation, the controllability time $T$ and
the control region $\omega$ can be chosen as small as one likes.

But this property of null controllability of the parabolic model is far from being true and well understood  when the model incorporates memory terms (see, for instance, \cite{G-I, Ha-Pan}).

When $a=0$ and under certain conditions,
 in  \cite{YZ1} it was shown that the system (\ref{1.2}) enjoys a finite speed of
propagation property for finite heat pulses, what makes it more realistic  for heat conduction. This has also important consequences from a control theoretical point of view.
For instance, when $a=0$,  under suitable
conditions on $a^{ij}(\cdot)$, geometric conditions on    $\omega$ and provided $T>0$ is large enough,  the system enjoys the control property that,
given $y_0, \,y_1\in L^2(\Omega)$, there is a control $u\in
L^2((0,T)\times\omega)$ such that the corresponding solution $y\in
C([0,T];\,L^2(\Omega))$ satisfies $y(T)=y_1$ in $\Omega$ (\cite{FYZ}).
We refer to \cite{Castro, LLTT, LiuY, Z} for some related works in this respect.

Note however that this result does not guarantee anything about the value of the accumulated memory that the system reaches at time $t=T$.  Accordingly, it  does not guarantee that the system may be driven to rest since this fact, in addition to the condition $y(T ) \equiv 0$, would require also that the memory term reaches the null value:
$$\int_0^T b(T-s,x)y_{x_i}(s,x)ds\equiv 0.
$$
In this sense, this  result has to be viewed as a property of partial controllability, but not of full controllability, since the later would require to control of the memory term too.

When $a=1$, (\ref{1.2}) is a controlled heat equation with a parabolic memory kernel. In this case, as it has been shown  in recent years (See \cite{G-I, Ha-Pan, Pandolfi, ZG}),  the null controllability may fail whenever the memory kernel $b(\cdot,\cdot)$ is a non-trivial constant and the control region $\omega$ is fixed, independent of time.  Nevertheless, the approximate controllability property is still possible for the same equation, at least for some special cases (See \cite{Barbu, ZG}).  The full picture  is still unclear and this papers aims to contribute in this direction.

The main results in this paper consist on, first, formulating the proper notion of controllability for these memory systems and then proving that, even if this property fails to be true for control supports that are independent of time, they hold provided the support of the control moves, covering the whole domain where the equation evolves, in the spirit of previous results in \cite{CRZ} for the system of viscoelasticity.

In order to illustrate this link between viscoelasticity and the memory models under consideration let us first analyze the simplest case:
\begin{equation}\label{heat0}
\left\{
\begin{array}{ll}
\displaystyle y_t-\Delta y + \int_0^t y(s)ds   = u\chi_{\omega}(x) &  \mbox{in}  \    Q,  \\[2mm]
y = 0 & \mbox{on} \  \Sigma, \\[2mm]
y(0) = y_0 & \mbox{in} \ \Omega.
\end{array}
\right.
\end{equation}
Setting $z(t)= \int_0^t y(s)ds$, this system can be rewritten as
\begin{equation}\label{heat00}
\left\{
\begin{array}{ll}
\displaystyle y_t-\Delta y + z   = u\chi_{\omega}(x) &  \mbox{in}  \    Q,  \\[2mm]
z_t=y &  \mbox{in}  \    Q,  \\[2mm]
y = z = 0 & \mbox{on} \  \Sigma, \\[2mm]
y(0) = y_0, \, z(0)=0& \mbox{in} \ \Omega.
\end{array}
\right.
\end{equation}
This system is constituted by the coupling of a heat equation with an ordinary differential equation (ODE), as in the context of viscoelasticity (See \cite{CRZ}). Of course the full null control of the system requires driving both the state $y$ and the memory term $z= \int_0^t y(s)ds$ to the null state at time $t=T$. But the presence of the ODE component makes the controllability of the system to be impossible if the control is confined to a strict subset $\omega$ of $\Omega$. This is why the support of the control needs to move to cover the domain where the equation evolves in the control time horizon.

As we shall see, the main ideas and techniques developed in \cite{CRZ} can be adapted to this setting.

To present our main results we consider the following abstract setting:
\begin{equation}\label{0.2}
 \left\{
 \begin{array}{ll}
 \ds y_t=A y+\int_0^tM(t-s)y(s)ds+B (t)u,\quad t\in (0,T],\\[2mm]
 \ds y(0)=y_0.
 \end{array}
 \right.
 \end{equation}
Here, $y=y(t)$ is the state variable which takes
values in a Hilbert space $Y$, $A$ generates a
$C_0-$semigroup $e^{A t}$ on $Y$, $M(\cdot)\in L^1(0,T;\mathcal L (Y))$, $u$ denotes the control variable taking values in another Hilbert
space $U$, and $B(\cdot)\in L^2(0,T;\mathcal L (U;Y))$.

It is easy to check that (e.g., \cite{YZ1}), under some mild assumptions on the coefficients $a^{ij}(\cdot)$ and $b(\cdot,\cdot)$, (\ref{1.2}) is a special case of \eqref{0.2}.

As in the context of  systems without memory terms,   \eqref{0.2} could be said to be null controllable if for any $y_0\in Y$, there exists a control $u(\cdot)\in L^2(0,T;U)$ such that the corresponding solution $y(\cdot)$ satisfies
$ y(T)=0$. Nevertheless,  because of the inertia effect of the memory term $\int_0^tM(t-s)y(s)ds$, the null state of \eqref{0.2} at $T$ cannot be kept for $t \ge T$ in the absence of control $u(t)=0$ for a.e. $t>T$. To guarantee this,  one needs to impose the  extra requirement that
  \begin{equation}\label{0zx.13}
 \int_0^TM(T-s)y(s)ds=0.
 \end{equation}

Going a bit further, we introduce the following concept of memory-type null controllability.

 \begin{definition}\label{1d2}
Given a memory kernel $\wt M(\cdot)\in L^1(0,T;\mathcal L (Y))$,  not necessarily  the same as $M(\cdot)$ in \eqref{0.2},
the equation \eqref{0.2} is called memory-type null controllable (with the memory kernel $\wt M(\cdot)$) if for any $y_0\in Y$, there is a control $u(\cdot)\in L^2(0,T;U)$ such that the corresponding solution $y(\cdot)$ satisfies
 \begin{equation}\label{0e1s}
 y(T)=0\ \ \ \hbox{ and }\ \ \ \int_0^T\wt M(T-s)y(s)ds=0.
 \end{equation}
 \end{definition}

The classical notion of (partial) null controllability of (\ref{0.2}) in the sense of (\ref{null}) is a special case of  memory-type null controllability of (\ref{0.2}) (by taking the memory kernel $\wt M(\cdot)\equiv 0$). But the full control of the system, as mentioned above, requires to take $\wt M \equiv M$.

The main goal of this article is to study the memory-type null controllability of (\ref{0.2}). By means of  classical duality arguments,  the problem will be reduced to the  obtention of suitable observability estimates for its adjoint system (see Proposition \ref{prop-2}). However, the required  observability estimates have not been addressed so far and this is the objective of the present paper,  focusing on finite-dimensional ordinal differential equations and  parabolic equations. As we shall see, in order to achieve the memory-type null controllability of (\ref{0.2}), except for some trivial cases, it is necessary to use controls with moving support. This is why we choose the control operator $B(\cdot)$ to be time-dependent (see Remark \ref{rek-2} for further explanations).

\begin{remark}\label{rek-1.2}
As in the classical setting of evolution equations without memory terms, one may introduce the (apparently stronger) condition of memory-type trajectory controllability in the sense that both the state $y$ and the memory at time $T$ match the values of a given trajectory of \eqref{0.2} and its corresponding memory by means of a suitable control $u(\cdot)$. Due to the linearity of the system under consideration, the memory-type trajectory controllability follows from its memory-type null controllability. In the sequel, accordingly, we shall focus on the problem of memory-type null controllability.
\end{remark}

The rest of this paper is organized as follows. In Section \ref{GC} we consider abstract evolution equations with memory terms and prove that the null and memory-type null controllability properties are equivalent to certain observability inequalities for appropriate adjoint systems. In Section \ref{FDC} we consider ordinary differential equations with memory and prove several rank conditions ensuring memory-type null controllability.  In Section \ref{ParabolicCase} we prove the memory-type null controllability for  parabolic equations with memory under suitable conditions on the moving control. Finally, in Section \ref{s6-1}, we list some open problems related to the topic in this paper.

\section{Abstract duality and the observability analog}\label{GC}

In this section, we consider the problem of memory-type null controllability (with the memory kernel $\wt M(\cdot)$) of (\ref{0.2}). For this, we introduce its adjoint system\footnote{Throughout this
paper, for any operator-valued function $R$, we denote by $R^*$ its
pointwise dual operator-valued function. For example, if $R\in
L^1(0,T; {\mathcal L}(Y))$, then
$R^*\in L^1(0,T; {\mathcal L}(Y))$,
and $|R|_{L^1(0,T; {\mathcal L}(Y))}=|R^*|_{L^1(0,T; {\mathcal L}(Y))}$.}:
\begin{equation}\label{adj1-z0}
\left\{
\begin{array}{ll}
\ds w_t=-A^* w -\int_t^TM(s-t)^* w(s)ds+\wt M(T-t)^*z_T, \quad t\in [0,T),  \\[2mm]
w(T) = w_T,
\end{array}
\right.
\end{equation}
where $w_T,z_T\in Y$.

We have the following result.

\begin{proposition}\label{prop-2}
Equation \eqref{0.2} is memory-type null controllable (with the memory kernel $\wt M(\cdot)$) if and only if there is a constant $C>0$ such that solutions of \eqref{adj1-z0} satisfy
 \begin{equation}\label{2e2.3}
 |w(0)|_Y^2\leq C\int_0^T|B(s)^*w(s)|_U^2ds,\quad\quad \forall\; w_T,z_T\in Y.
 \end{equation}
\end{proposition}

\begin{proof} The proof is standard. For the readers' convenience, we give the details below.

We prove first the ``if" part. Fix a $y_0\in Y$. We introduce a linear subspace $\mathcal{L}$ of $L^2(0,T;U)$ as follows:
 $$
  \mathcal{L}=\{B(\cdot)^*w(\cdot)\;|\; w(\cdot)\hbox{ solves }\eqref{adj1-z0}\hbox{ for some } w_T,z_T\in Y\}.
  $$
 For any $B(\cdot)^*w(\cdot) \in \mathcal{L}$, we define
  $$
  \mathfrak{F}\big(B(\cdot)^*w(\cdot)\big)=-(w(0),y_0)_Y.
  $$
 By \eqref{2e2.3}, we see that $\mathfrak{F}$ is a bounded linear functional on the normed vector space $\mathcal{L}$ (with the norm inherited from $L^2(0,T;U)$). Hence, by the Hahn-Banach Theorem, $\mathfrak{F}$ can be extended to a bounded linear functional on $L^2(0,T;U)$). Now, Riesz Representation
Theorem allows us to find a function $\eta(\cdot)\in L^2(0,T;U)$ such that
 \begin{equation}\label{2e2z3}
 \int_0^T(B(t)^*w(t),\eta(t))_Udt=-(w(0),y_0)_Y.
  \end{equation}

We claim that
 \begin{equation}\label{2z3}
 u(\cdot)=\eta(\cdot)
 \end{equation}
is the desired control. Indeed, for any $w_T,z_T\in Y$, by \eqref{0.2} and \eqref{adj1-z0}, we
obtain that
 \begin{equation}\label{2z23}
 \begin{array}{ll}
 \displaystyle (w_T,y(T))_Y-(w(0),y_0)_Y=\int_0^T\frac{d}{dt}(w,y)_Y=\int_0^T[(w_t,y)_Y+(w,y_t)_Y]dt\\[2mm]
 \displaystyle=\int_0^T\left[\left(-A^* w -\int_t^TM(s-t)^* w(s)ds+\wt M(T-t)^*z_T,y\right)_Y\right.\\[2mm]
 \displaystyle\qquad\qquad+\left.\left(w,A y+\int_0^tM(t-s)y(s)ds+B(t) u\right)_Y\right]dt
 \\[2mm]
 \displaystyle=\left(z_T,\int_0^T\wt M(T-t)y(t)dt\right)_Y+\int_0^T(B(t)^*w(t),u(t))_Udt.
 \end{array}
 \end{equation}
Combining \eqref{2e2z3}, \eqref{2z3} and \eqref{2z23}, we end up with
 $$
 (w_T,y(T))_Y-\left(z_T,\int_0^T\wt M(T-t)y(t)dt\right)_Y=0,\qquad\forall\; w_T,z_T\in Y.
  $$
 Hence $y(T)=\int_0^T\wt M(T-t)y(t)dt=0$, as desired.

 Next, we prove the ``only if" part. For any $w_T,z_T\in Y$, by the equation \eqref{adj1-z0}, we may define a bounded linear operator ${\mathcal F}: Y\times Y\to Y$ as follows:
  \begin{equation}\label{2xfxe2.3}
 {\mathcal F}(w_T,z_T)=w(0).
  \end{equation}
 We now use the contradiction argument to prove \eqref{2e2.3}. Assume that \eqref{2e2.3} was not true. Then, one could find two sequences $\{z_T^k\}_{k=1}^\infty,\{w_T^k\}_{k=1}^\infty\subset Y$ such that the corresponding solutions $w^k(\cdot)$ to \eqref{adj1-z0}  (with $(w_T,z_T)$ replaced by $(w_T^k,z_T^k)$) satisfy
  \begin{equation}\label{2xxe2.3}
 0\leq \int_0^T|B(s)^*w^k(s)|_U^2ds<\frac{1}{k^2}|w^k(0)|_Y^2,\quad\quad \forall\;k\in \mathbb{N}.
 \end{equation}
Write
 $$
 \tilde w_T^k=\sqrt{k}\frac{w_T^k}{|w^k(0)|_Y}, \quad \tilde z_T^k=\sqrt{k}\frac{z_T^k}{|w^k(0)|_Y},
 $$
and denote by $\tilde w^k(\cdot)$ the corresponding solution to \eqref{adj1-z0}  (with $(w_T,z_T)$ replaced by $(\tilde w_T^k,\tilde z_T^k)$). Then, it follows from \eqref{2xfxe2.3} and \eqref{2xxe2.3} that, for each $k\in \mathbb{N}$,
  \begin{equation}\label{2xxe2.12}
 \int_0^T|B(s)^*\tilde w^k(s)|_U^2ds<\frac{1}{k},\quad\quad |{\mathcal F}(\tilde w_T^k,\tilde z_T^k)|_Y=\sqrt{k}.
 \end{equation}

Since  \eqref{0.2} is assumed to be memory-type null controllable (with the memory kernel $\wt M(\cdot)$), for any $y_0\in Y$, one can find a control $u(\cdot)\in L^2(0,T;U)$ such that the corresponding solution $y(\cdot)$ satisfies \eqref{0e1s}. For any $w_T,z_T\in Y$, by \eqref{0.2} and \eqref{adj1-z0}, similar to the proof of \eqref{2z23} and noting \eqref{0e1s}, we have
 $$
-(w(0),y_0)_Y=\int_0^T(B(t)^*w(t),u(t))_Udt.
 $$
In particular, it holds that
  \begin{equation}\label{2zaa23}
-({\mathcal F}(\tilde w_T^k,\tilde z_T^k),y_0)_Y=\int_0^T(B(t)^*\tilde w^k(t),u(t))_Udt.
 \end{equation}

By \eqref{2zaa23} and the first inequality in \eqref{2xxe2.12}, it is easy to see that ${\mathcal F}(\tilde w_T^k,\tilde z_T^k)$ tends to $0$ weakly in $Y$. Hence, by the Principle of Uniform Boundedness, we see that the sequence $\{{\mathcal F}(\tilde w_T^k,\tilde z_T^k)\}_{k=1}^\infty$ is uniformly bounded in $Y$,
contradicting the second equality in \eqref{2xxe2.12}.
This completes the proof of Proposition \ref{prop-2}.
\end{proof}

\begin{remark}\label{rek-2}
 Proposition \ref{prop-2} characterizes the property of memory-type null controllability in terms of a non-standard unique continuation property and observability inequality (\ref{2e2.3}) which, as we shall see, it is very hard to achieve when the control operator $B$ is time-independent, except for the trivial case where $B$ (from $U$ to $Y$) is onto. In particular, the results in \cite{G-I}, for instance, by means of a spectral analysis of the problem, show that this inequality may not hold for the heat equation with memory terms, if the support of the control is independent of time.

 This is why, in practice, our sufficient conditions for memory-type controllability will require the control operator $B$ to depend on time. This is particularly natural when dealing with concrete PDEs, and when the control operator $B$ localises the action of the support in a subdomain of $\Omega$. Accordingly, as in the context of viscoelasticity (see \cite{CRZ}), considering moving controls is a natural way of getting rid of the lack of the strong observability inequality (\ref{2e2.3}) .
\end{remark}
\begin{remark}\label{rek-2b}
The observability inequality in (\ref{2e2.3}) is relevant not only because it provides a characterisation of the property of memory-type null controllability but also because it leads to a constructive algorithm for control as explained in \cite{Zua2} in the PDE setting.
Indeed, assuming that
the observability inequality in (\ref{2e2.3}) holds, let us consider the following quadratic functional defined on the solutions to the adjoint system (\ref{adj1-z0}):
\begin{equation}
J(w_T,z_T) = \frac{1}{2} \int_0^T|B(s)^*w(s)|_U^2ds +(w(0),y_0)_Y.
\end{equation}
In principle $J$ is defined for $(w_T, z_T) \in Y \times Y$, and it is a continuous and convex functional in  that space. Let us assume that $J$ achieves its minimum at some $(w_T^*,z_T^*)$. It is then easy to see that the control $u= B^* w^*$, $w^*$ being the solution of the adjoint system corresponding to the minimiser, is the control we are looking for, ensuring the control condition  (\ref{0e1s}).

It is however important to observe that the existence of the minimiser is not a trivial issue. Indeed, the observability inequality  (\ref{2e2.3}) is very weak since it only leads to an upper bound on the norm of $w(0)$ in $Y$ but not on $w_T$,  neither on $z_T$. Thus, in order to minimize $J$ we need to introduce the Hilbert space closure of $Y\times Y$ with respect to the Hilbertian norm defined by
\begin{equation}\label{norm}
\left[\int_0^T|B(t)^*w(t)|^2 dt\right]^{1/2}.
\end{equation}

Note that whether the above semi-norm actually defines a norm is not a trivial fact.

This issue is well understood in the context of PDE (see \cite{Zua2}). For the wave equation without memory terms, the corresponding adjoint system (\ref{adj1-z0}) with $z_T \equiv 0$ being time-reversible, the observation of the norm of $w(0)$ in $Y$ ensures also the observation of $w_T$. The observability inequality allows then to minimise the functional $J$ (that would be now independent of $z_T$) with respect to $w_T$ in $Y$.

In the case of the heat equation the issue is more subtle since an estimate on $w(0)$ in $Y$ does not imply an estimate of $w_T$ in $Y$. However, by the backward uniqueness property of parabolic equations, this allows to define the completion of $Y$ with respect to the norm (\ref{norm}) and to ensure the existence of the minimiser of $J$ with respect to $w_T$ in that space.

In the present context of memory-type null controllability, despite the characterisation of the controllability property in terms of the observability of the augmented adjoint system, the actual implementation of this variational method to build controls needs further clarification.

Note that this kind of characterisation is of use in different contexts, and in particular in order to build efficient numerical approximation procedures by means of gradient descent methods.

\end{remark}

\section{The finite dimensional case}\label{FDC}

In this section, we consider the following controlled ordinary differential equation with a memory term:
 \begin{equation}\label{0.3}
 \left\{
 \begin{array}{ll}
 \ds y_t=A y+\int_0^tM(t-s)y(s)ds+B u,\quad t\in (0,T],\\[2mm]
 y(0)=y_0.
 \end{array}
 \right.
 \end{equation}
Here, $y=y(t)$ is the state variable which takes
values in $\mathbb{R}^n$, $A\in \mathbb{R}^{n\times n}$, $M(\cdot)\in L^1(0,T;\mathbb{R}^{n\times n})$, $u$ denotes the control variable taking values in $\mathbb{R}^m$ ($m\in \mathbb{N}$), and $B\in \mathbb{R}^{n\times m}$. In the sequel, we denote by $K^\top$ the transpose of a matrix $K\in \mathbb{R}^{n\times m}$.

According to the previous section, fix a memory kernel $\wt M(\cdot)\in L^1(0,T;\mathbb{R}^{n\times n})$, we then need to consider the following adjoint system:
\begin{equation}\label{adj0zx-0}
\left\{
\begin{array}{ll}
\ds w_t=-A^\top  w -\int_t^T M(s-t)^\top w(s)ds+\wt M(T-t)^\top z_T, \quad t\in [0,T),  \\[2mm]
w(T) = w_T,
\end{array}
\right.
\end{equation}
where $ w_T, z_T\in \mathbb{R}^n$.

We have the following result.

\begin{theorem}\label{EC-FD}
(i) If $M(\cdot), \wt M(\cdot)\in L^1(0,T;\mathbb{R}^{n\times n})$, and for any solution $w$ to the equation \eqref{adj0zx-0},
\begin{equation} \label{UC-1}
B^\top w \equiv 0 \quad \hbox{in } [0,T] \Rightarrow w_T =\wt M(t)^\top z_T =0, \quad\hbox{ a.e. }t \in [0,T],
\end{equation}
then the equation \eqref{0.3} is memory-type null controllable;

(ii) If $M(\cdot)= G\wt M(\cdot)$ and $\wt M'(\cdot)= \wt G\wt M(\cdot)$ for some (constant matrices) $G, \wt G\in \mathbb{R}^{n\times n}$ and the equation \eqref{0.3} is memory-type null controllable, then for any solution to the equation \eqref{adj0zx-0}, it holds that
 \begin{equation} \label{UC-2}
B^\top w \equiv 0 \quad \hbox{in } [0,T] \Rightarrow w_T =\wt M(t)^\top z_T =0, \quad t \in [0,T].
\end{equation}
\end{theorem}

\begin{proof}
To prove (i), by (\ref{UC-1}) and using the classical compactness-uniqueness argument, it follows that solutions to \eqref{adj0zx-0} satisfy
 \begin{equation}\label{zx-0-2}
|w_T|^2+\left(\int_0^T| \wt M(t)^\top z_T|dt\right)^2\le C\int_0^T|B^\top w(t)|^2dt,\quad\forall\;w_T, z_T\in \mathbb{R}^n.
\end{equation}
Applying the usual energy estimate to \eqref{adj0zx-0}, we have
  \begin{equation}\label{zx-0-3}
|w(0)|^2\le C\left[|w_T|^2+\left(\int_0^T| \wt M(t)^\top z_T|dt\right)^2\right],\quad\forall\;w_T, z_T\in \mathbb{R}^n.
\end{equation}
Combining (\ref{zx-0-2}) and (\ref{zx-0-3}), we arrive at
 \begin{equation}\label{zx-gg0-3}
 |w(0)|^2\le C\int_0^T|B^\top w(t)|^2dt,\quad\forall\;w_T, z_T\in \mathbb{R}^n.
 \end{equation}
Hence, Proposition \ref{prop-2} implies that \eqref{0.3} is memory-type null controllable.

\medskip

We now prove (ii). By Proposition \ref{prop-2} and the memory-type null controllability of \eqref{0.3}, we see that solutions to the equation \eqref{adj0zx-0} satisfy \eqref{zx-gg0-3}. Hence,  $w(0)=0$ as a consequence of $B^\top w \equiv 0$ in $[0,T]$. Write $\varphi=w_t$. By the first equation of (\ref{adj0zx-0}), noting $M(\cdot)= G\wt M(\cdot)$ and $\wt M'(\cdot)= \wt G\wt M(\cdot)$, we have
 $$
 \begin{array}{ll}
 \displaystyle\varphi_t&\displaystyle=-A^\top  \varphi +M(0)^\top w+\int_t^TM'(s-t)^\top w(s)ds-\wt M'(T-t)^\top z_T\\[2mm]
 &\displaystyle=-A^\top  \varphi +M(0)^\top w+M(s-t)^\top w(s)\Big|_{s=t}^{s=T}-\int_t^TM(s-t)^\top \varphi(s)ds-\wt M'(T-t)^\top z_T\\[2mm]
 &\displaystyle=-A^\top  \varphi-\int_t^T M(s-t)^\top \varphi(s)ds + \wt M(T-t)^\top (G^\top w_T-\wt G^\top z_T).
 \end{array}
 $$
Hence, $\varphi$ solves
\begin{equation}\label{adj0zx-0-2-2}
\left\{
\begin{array}{ll}
\ds \varphi_{t}=-A^\top  \varphi-\int_t^T M(s-t)^\top \varphi(s)ds + \wt M(T-t)^\top (G^\top w_T-\wt G^\top z_T), \quad t\in [0,T),  \\[2mm]
\varphi(T) = -A^\top  w_T + \wt M(0)^\top z_T.
\end{array}
\right.
\end{equation}
Noticing that \eqref{adj0zx-0-2-2} is of the form \eqref{adj0zx-0}, it follows from \eqref{zx-gg0-3} that
\begin{equation}\label{FDNUO}
|w_t(0)|^2= |\varphi(0)|^2\leq C\int_0^T|B^\top \varphi(s)|^2ds=C\int_0^T|B^\top w_t(s)|^2ds=0.
 \end{equation}
Hence, $w_t(0)=0$. Repeating this argument, we see that $\left.d^kw(t)/dt^k\right|_{t=0}=0$ for all $k=0, 1,2,\cdots$. Since $w(t)$ is analytic in time $t$, it follows that $w(\cdot)\equiv 0$ in $[0,T]$.  Hence $w_T =\wt M(t)^\top z_T =0$ for any $t\in [0,T]$.
 \end{proof}

As an immediate consequence of Theorem \ref{EC-FD}, we have the following result.

\begin{corollary}\label{EC-cFD}
 Assume that $M(\cdot) \equiv M\in \mathbb{R}^{n\times n}$, $\wt M(\cdot) \equiv\wt M\in \mathbb{R}^{n\times n}$, $M= G\wt M$ for some $G\in \mathbb{R}^{n\times n}$. Then, the equation \eqref{0.3} is memory-type null controllable (with the kernel $\wt M$) if and only if solutions to the equation \eqref{adj0zx-0} satisfy 
 $$
B^\top w \equiv 0 \quad \hbox{in } [0,T] \Rightarrow w_T =\wt M^\top z_T =0.
 $$
\end{corollary}

We now present some rank conditions for the memory-type null controllability of \eqref{0.3}

\begin{theorem}\label{EC-2FD}
(i) Assume that $M(\cdot), \wt M(\cdot)\in L^1(0,T;\mathbb{R}^{n\times n})\cap C^\infty([0,T);\mathbb{R}^{n\times n})$, and define $A_i$, $M_i(\cdot)$ and $\wt M_i(\cdot)$ ($i=1,2,\cdots$) inductively by
 \begin{equation}\label{m1az2}
   \left\{
   \begin{array}{ll}
 A_{i+1}=AA_i+M_i(0),\quad  M_{i+1}(\cdot)=M(\cdot)A_i+M_i'(\cdot),\quad \wt M_{i+1}(\cdot)=\wt M(\cdot)A_i+\wt M_i'(\cdot)\\[2mm]
 A_1=A,\qquad M_1(\cdot)=M(\cdot),\qquad \wt M_1(\cdot)=\wt M(\cdot).
 \end{array}
 \right.
 \end{equation}
 If$\;$\footnote{Note that $A_{i+1}$  in \eqref{m1az2} is time-independent even if $M_{i+1}(\cdot)$  depends on $t$. Because of this, \eqref{2zxeds.3} is an algebraic condition.}
  \begin{equation}\label{2zxeds.3}
 \rank \left(\begin{array}{ccccccc}B & A_1B & A_2B &\cdots &A_iB &A_{i+1}B &\cdots\\[2mm]
 0 & \wt M_1(0)B & \wt M_2(0)B &\cdots &\wt M_i(0)B &\wt M_{i+1}(0)B &\cdots
 \end{array}\right)=2n,
 \end{equation}
then the equation \eqref{0.3} is memory-type null controllable;

(ii) Assume that both $M(\cdot)$ and $\wt M(\cdot)$ are analytic in $[0,T]$, and define $A_i$ as that in
 \eqref{m1az2} ($i=1,2\cdots$). Assume that $\left.d^i\wt M(t)/dt^i\right|_{t=0}=\wt M(0)G_i$ for some $G_i\in \mathbb{R}^{n\times n}$, and define
  \begin{equation}\label{2zads.3}
  F_i=A_i+G_1A_{i-1}+\cdots+G_{i-1}A_1+G_i.
  \end{equation}
 If
  \begin{equation}\label{2z-xeds.3}
 \rank \left(\begin{array}{ccccccc}B & A_1B & A_2B &\cdots &A_iB &A_{i+1}B &\cdots\\[2mm]
 0 &  B &  F_1B &\cdots &F_{i-1} B & F_iB &\cdots
 \end{array}\right)=2n,
 \end{equation}
then the equation \eqref{0.3} is memory-type null controllable;

(iii) Assume that $M(\cdot) \equiv M\in \mathbb{R}^{n\times n}$ and $\wt M(\cdot) \equiv\wt M\in \mathbb{R}^{n\times n}$, and define $A_i$ ($i=1,2\cdots$) inductively by
\begin{equation}\label{m1z2}
   \left\{
   \begin{array}{ll}
 A_{i+1}=AA_i+M_i,\qquad  M_{i+1}=MA_i\\[2mm]
 A_1=A,\qquad M_1=M.
 \end{array}
 \right.
 \end{equation}
 If
 \begin{equation}\label{2aaz-xeds.3}
 \rank \left(\begin{array}{ccccccc}B & A_1B & A_2B &\cdots &A_{2n+1}B \\[2mm]
 0 &  B &  A_1B &\cdots & A_{2n}B
 \end{array}\right)=2n,
 \end{equation}
then the equation \eqref{0.3} is memory-type null controllable. If, additionally, $\det \wt M\not=0$, then the condition \eqref{2aaz-xeds.3} is also necessary for \eqref{0.3} to be memory-type null controllable.
\end{theorem}

\begin{proof}
(i) Suppose that for some $w_T,z_T\in \mathbb{R}^n$, the corresponding solution $w(\cdot)$ of (\ref{adj0zx-0}) satisfies $B^\top w(\cdot)\equiv 0$ in $[0,T]$.

It is easy to see that $B^\top w(\cdot)\equiv 0$ in $[0,T]$ gives
 \begin{equation}\label{zx1m}
   B^\top w_T=0,
 \end{equation}
and
 \begin{equation}\label{zx2m}
   0=-B^\top w_t=B^\top \left(A_1^\top  w + \int_t^TM_1(s-t)^\top  w(s)ds-\wt M_1(T-t)^\top z_T\right)\quad \hbox{in } [0,T].
 \end{equation}

By (\ref{zx2m}) and using the equation (\ref{adj0zx-0}), we find that
 \begin{equation}\label{zx3m}
   B^\top A_1^\top w_T-B^\top \wt M_1(0)^\top z_T=0,
 \end{equation}
and
\begin{equation}\label{zx4m}
 \begin{array}{ll}
   \ds w_{tt}&\ds=-A_1^\top  w_t +M_1(0)^\top  w+\int_t^TM_1'(s-t)^\top  w(s)ds-\wt M_1'(T-t)^\top z_T\\[2mm]
   &\ds=A_1^\top  A^\top w +\int_t^TA_1^\top M(s-t)^\top  w(s)ds-A_1^\top \wt M(T-t)^\top z_T\\[2mm]
   &\quad\ds+M_1(0)^\top  w+\int_t^TM_1'(s-t)^\top  w(s)ds-\wt M_1'(T-t)^\top z_T\\[2mm]
   &\ds =A_2^\top  w + \int_t^T M_2(s-t)^\top w(s)ds-\wt M_2(T-t)^\top z_T \quad \hbox{in } [0,T].
 \end{array}
\end{equation}

More generally, we have
 \begin{equation}\label{zxim}
   B^\top A_i^\top w_T-B^\top \wt M_i(0)^\top z_T=0,
 \end{equation}
and
\begin{equation}\label{zxi=1m}
   \frac{d^{i+1}w}{dt^{i+1}}=(-1)^{i+1}A_{i+1}^\top  w + (-1)^{i+1}\int_t^T M_{i+1}(s-t)^\top w(s)ds+(-1)^i\wt M_{i+1}(T-t)^\top z_T \quad \hbox{in } [0,T].
\end{equation}

By (\ref{zx1m}), (\ref{zx3m}) and (\ref{zxim}), we end up with
 \begin{equation}\label{zxi0m}
 (w_T^\top,-z_T^\top)\left(\begin{array}{ccccccc}B & A_1B & A_2B &\cdots &A_iB &A_{i+1}B &\cdots\\[2mm]
 0 & \wt M_1(0)B & \wt M_2(0)B &\cdots &\wt M_i(0)B &\wt M_{i+1}(0)B &\cdots
 \end{array}\right)=0.
 \end{equation}
By (\ref{2zxeds.3}) and (\ref{zxi0m}), we conclude that $w_T=z_T=0$. Hence, by the first conclusion of Theorem \ref{EC-FD}, we conclude that  \eqref{0.3} is memory-type null controllable.

\medskip

(ii) As in (i), we suppose that for some $w_T,z_T\in \mathbb{R}^n$, the corresponding solution $w(\cdot)$ of (\ref{adj0zx-0}) satisfies $B^\top w(\cdot)\equiv 0$ in $[0,T]$. Then, we have (\ref{zx1m}). By (\ref{m1az2}), \eqref{2zads.3} and $\left.d^i\wt M(t)/dt^i\right|_{t=0}=\wt M(0)G_i$, it is easy to check that
 \begin{equation}\label{z-xim}
 \left\{
  \begin{array}{ll}
  \ds\wt M_1(0)=\wt M(0),\\[2mm]\ds
  \wt M_{i+1}(0) =\wt M(0)(A_i+G_1A_{i-1}+\cdots+G_{i-1}A_1+G_i)=\wt M(0)F_i,\quad i=1,2,\cdots.
  \end{array}
  \right.
 \end{equation}
Hence, (\ref{zx3m}) reads
 \begin{equation}\label{zx--3m}
   B^\top A_1^\top w_T-B^\top \wt M(0)^\top z_T=0,
 \end{equation}
and (\ref{zxim}) is specialized as
\begin{equation}\label{zxi00m}
   B^\top A_i^\top w_T-B^\top F_{i-1}^\top \wt M(0)^\top z_T=0, \quad i=2,3,\cdots.
 \end{equation}

By (\ref{zx1m}), (\ref{zx--3m}) and (\ref{zxi00m}), we obtain that
 \begin{equation}\label{zxddi0m}
 (w_T^\top,-z_T^\top \wt M(0))\left(\begin{array}{ccccccc}B & A_1B & A_2B &\cdots &A_{i+1}B &A_{i+2}B &\cdots\\[2mm]
 0 & B & F_1B &\cdots &F_iB &F_{i+1}B &\cdots
 \end{array}\right)=0.
 \end{equation}
By (\ref{2z-xeds.3}) and (\ref{zxddi0m}), we see that
 \begin{equation}\label{zxdd110m}
 w_T=\wt M(0)^\top z_T=0.
 \end{equation}
By (\ref{zxi=1m}), (\ref{z-xim}) and (\ref{zxdd110m}), it follows that
\begin{equation}\label{zxi=11m}
  \left. \frac{d^{i+1}w}{dt^{i+1}}\right|_{t=T}=(-1)^{i+1}A_{i+1}^\top  w_T +(-1)^iF_i^\top \wt M(0)^\top z_T=0.
\end{equation}

Since both $M(\cdot)$ and $\wt M(\cdot)$ are analytic in $[0,T]$, so is $w(\cdot)$. By (\ref{zxi=11m}), we conclude that $w(\cdot)\equiv 0$ in $[0,T]$. Hence, by the first equation in (\ref{adj0zx-0}), $\wt M(t)^\top z_T =0$  in $[0,T]$. Now, by the first conclusion of Theorem \ref{EC-FD}, the equation  \eqref{0.3} is memory-type null controllable.

\medskip

(iii) We can use the result in (ii). For the present case, it is easy to check that the $F_i$ defined by (\ref{2zads.3}) is specialized to $F_i=A_i$. We claim that
 \begin{equation}\label{2zranxeds.3}
 \begin{array}{ll}
 \rank \left(\begin{array}{ccccccc}B & A_1B & A_2B &\cdots &A_iB &A_{i+1}B &\cdots\\[2mm]
 0 &  B &  A_1B &\cdots &A_{i-1} B & A_iB &\cdots
 \end{array}\right)\\[2mm]
 =\rank\left(\begin{array}{ccccccc}B & A_1B & A_2B &\cdots &A_{2n+1}B \\[2mm]
 0 &  B &  A_1B &\cdots & A_{2n}B
 \end{array}\right).
 \end{array}
 \end{equation}

It is easy to see that (\ref{m1z2}) is a special case of (\ref{m1az2}). From (\ref{m1z2}), we see that
  $$
   \left(
   \begin{array}{ll}
   A_{i+1}\\[2mm]
   M_{i+1}
   \end{array}
   \right)=
   \left(
   \begin{array}{ll}
   A&I_n\\[2mm]
   M&0
   \end{array}
   \right)
   \left(
   \begin{array}{ll}
   A_i\\[2mm]
   M_i
   \end{array}
   \right).
  $$
Hence,
 \begin{equation}\label{x1}
      \left(
   \begin{array}{ll}
   A_{i+1}\\[2mm]
   M_{i+1}
   \end{array}
   \right)=
   \left(
   \begin{array}{ll}
   A&I_n\\[2mm]
   M&0
   \end{array}
   \right)^i
   \left(
   \begin{array}{ll}
   A\\[2mm]
   M
   \end{array}
   \right).
 \end{equation}

Denote by $\lambda^{2n}+a_1\lambda^{2n-1}+a_2\lambda^{2n-2}+\cdots+a_{2n}$ the characteristic polynomial of $\left(
   \begin{array}{ll}
   A&I_n\\[2mm]
   M&0
   \end{array}
   \right)
 $, where $a_1,a_2,\cdots,a_{2n}\in \mathbb{R}$. By the Hamilton-Cayley theorem, it follows that
  \begin{equation}\label{x2}
     \left(
   \begin{array}{ll}
   A&I_n\\[2mm]
   M&0
   \end{array}
   \right)^{2n}+a_1\left(
   \begin{array}{ll}
   A&I_n\\[2mm]
   M&0
   \end{array}
   \right)^{2n-1}+a_2\left(
   \begin{array}{ll}
   A&I_n\\[2mm]
   M&0
   \end{array}
   \right)^{2n-2}+\cdots+a_{2n}I_{2n}=0.
    \end{equation}
Combining \eqref{x1} and \eqref{x2}, we have
  $$
  \begin{array}{ll}
   \ds\left(
   \begin{array}{ll}
   A_{2n+1}\\[2mm]
   M_{2n+1}
   \end{array}
   \right)=\left(
   \begin{array}{ll}
   A&I_n\\[2mm]
   M&0
   \end{array}
   \right)^{2n}
   \left(
   \begin{array}{ll}
   A\\[2mm]
   M
   \end{array}
   \right)\\[5mm]
   =-a_1\left(
   \begin{array}{ll}
   A&I_n\\[2mm]
   M&0
   \end{array}
   \right)^{2n-1}
   \left(
   \begin{array}{ll}
   A\\[2mm]
   M
   \end{array}
   \right)-a_2\left(
   \begin{array}{ll}
   A&I_n\\[2mm]
   M&0
   \end{array}
   \right)^{2n-2}
   \left(
   \begin{array}{ll}
   A\\[2mm]
   M
   \end{array}
   \right)-\cdots-a_{2n}\left(
   \begin{array}{ll}
   A\\[2mm]
   M
   \end{array}
   \right)\\[5mm]
   \ds=-a_1\left(
   \begin{array}{ll}
   A_{2n}\\[2mm]
   M_{2n}
   \end{array}
   \right)-a_2\left(
   \begin{array}{ll}
   A_{2n-1}\\[2mm]
   M_{2n-1}
   \end{array}
   \right)-\cdots-a_{2n}\left(
   \begin{array}{ll}
   A_1\\[2mm]
   M_1
   \end{array}
   \right).
   \end{array}
  $$
This gives
 \begin{equation}\label{x3}
 A_{2n+1}=-a_1A_{2n}-a_2A_{2n-1}-\cdots-a_{2n}A_1.
  \end{equation}
Similarly, 
 \begin{equation}\label{xsd3}
 A_{2n+2}=-a_1A_{2n+1}-a_2A_{2n}-\cdots-a_{2n}A_2.
  \end{equation}
Combining (\ref{x3}) and (\ref{xsd3}), we find that
 \begin{equation}\label{xs--d3}
 \left(\begin{array}{ll}A_{2n+2}\\[2mm] A_{2n+1}\end{array}
   \right)=-a_1\left(\begin{array}{ll}A_{2n+1}\\[2mm] A_{2n}\end{array}
   \right)-a_2\left(\begin{array}{ll}A_{2n}\\[2mm] A_{2n-1}\end{array}
   \right)-\cdots-a_{2n}\left(\begin{array}{ll}A_2\\[2mm] A_1\end{array}
   \right).
  \end{equation}
Inductively, from \eqref{xs--d3}, one can show that each $\left(\begin{array}{ll}A_{k+1}\\[2mm] A_k\end{array}
   \right)$ ($k\ge 2n+1$) can be expressed as a linear combination of $\left(\begin{array}{ll}A_2\\[2mm] A_1\end{array}
   \right),\left(\begin{array}{ll}A_3\\[2mm] A_2\end{array}
   \right),\cdots, \left(\begin{array}{ll}A_{2n+1}\\[2mm] A_{2n}\end{array}
   \right)$. Consequently, (\ref{2zranxeds.3}) is verified.

By the result in (ii) and (\ref{2zranxeds.3}), it is easy to see that under the condition
\eqref{2aaz-xeds.3}, the equation \eqref{0.3} is memory-type null controllable.

If, additionally,  $\det \wt M\not=0$, then, we use the contradiction argument to show that the condition \eqref{2aaz-xeds.3} is necessary for \eqref{0.3} to be memory-type null controllable.  Assume that the equation (\ref{0.3}) is memory-type null controllable but the condition \eqref{2aaz-xeds.3} does not hold. Then, in view of \eqref{2zranxeds.3},
 $$\rank \left(\begin{array}{ccccccc}B & A_1B & A_2B &\cdots &A_iB &A_{i+1}B &\cdots\\[2mm]
 0 &  B &  A_1B &\cdots &A_{i-1} B & A_iB &\cdots
 \end{array}\right)<2n.
 $$
This implies that there is a $(w_T,z_T)\in \mathbb{R}^{2n}\setminus\{0\}$ satisfying
 \begin{equation}\label{zaai0m}
 (w_T^\top,-z_T^\top \wt M)\left(\begin{array}{ccccccc}B & A_1B & A_2B &\cdots &A_iB &A_{i+1}B &\cdots\\[2mm]
 0 &  B &  A_1B &\cdots &A_{i-1} B & A_iB &\cdots
 \end{array}\right)=0.
 \end{equation}
Clearly, this $(w_T,z_T)$ satisfies
 $$
 \left\{\begin{array}{ll}
 B^\top w_T=0, \\[2mm] B^\top A_1^\top w_T-B^\top \wt M^\top z_T=0,\\[2mm]
 B^\top A_i^\top w_T-B^\top A_{i-1}^\top \wt M^\top z_T=0, \qquad i=2,3,\cdots.
 \end{array}\right.
 $$
Hence, the corresponding solution $w(\cdot)$ of (\ref{adj0zx-0}) satisfies
 \begin{equation}\label{zxk}
 \left.\frac{d^kB^\top w(t)}{dt^k}\right|_{t=T}=0,\ \ \quad k=1,2,\cdots.
 \end{equation}

Since $B^\top w(\cdot)$ is an analytic function, (\ref{zxk}) implies that $B^\top w(\cdot)\equiv 0$. In view of Theorem \ref{EC-FD}, this leads to $w_T=z_T=0$,  a contradiction.
\end{proof}

\begin{remark}
It is possible to consider the memory-type null controllability problem for the following general system:
\begin{equation}\label{0.3-0}
 \left\{
 \begin{array}{ll}
 \ds y_t=A(t) y+\int_0^tM(t-s)y(s)ds+B(t) u,\quad t\in (0,T],\\[2mm]
 y(0)=y_0,
 \end{array}
 \right.
 \end{equation}
where $A \in L^{\infty}(0,T; \mathbb{R}^{n\times n})$ and  $B \in L^{\infty}(0,T; \mathbb{R}^{n\times m})$.  In this case, one can work with the extended system
\begin{equation}\label{0.3-1}
 \left\{
 \begin{array}{ll}
 \ds y_t=A(t) y+\int_0^tM(t-s)y(s)ds+B(t) u,\quad t\in (0,T],\\[5mm]
\ds z_t = \wt M(0)y + \int_0^t\wt M'(t-s)y(s)ds,\quad t\in (0,T], \\[5mm]
 y(0)=y_0, \ z(0)=0,
 \end{array}
 \right.
 \end{equation}
and the condition \eqref{0e1s} is equivalent to
$
y(T)= z(T) = 0$.

When both $M(\cdot)$ and $\wt M(\cdot)$ are constant matrices, i.e., $M(\cdot) \equiv M\in \mathbb{R}^{n\times n}$ and $\wt M(\cdot) \equiv\wt M\in \mathbb{R}^{n\times n}$, and  $M = G\wt M$ for some $G \in \mathbb{R}^{n\times n}$, the system \eqref{0.3-1} can be rewritten as
\begin{equation}\label{0.3-2}
 \left\{
 \begin{array}{ll}
 \ds y_t=A(t) y+Gz+B(t) u,\quad t\in (0,T],\\[2mm]
 z_t = \wt My,\quad t\in (0,T], \\[2mm]
 y(0)=y_0, \ z(0)=0,
 \end{array}
 \right.
 \end{equation}
and the controllability of the system \eqref{0.3-2} can be analyzed by means of the  Silverman-Meadows condition (See, for instance, \cite[Theorem 1.18, p. 11]{Coron}). However, for general kernels $M(\cdot)$ and $ \wt M(\cdot)$,  even if $A(\cdot) = A$ is a constant matrix, we do not know how to obtain a Silverman-Meadows condition for system \eqref{0.3-1}. In this sense, Theorem \ref{EC-2FD} can be viewed as a memory-type variant of the Silverman-Meadows condition for the case where both $A(\cdot) = A$ and $B(\cdot) = B$ are constant matrices.

\end{remark}


%
%

\section{Memory-type null controllability of parabolic equations}\label{ParabolicCase}

In this section, we analyze the memory-type null controllability for parabolic equations.

We begin with the following heat equation with a memory term and a fixed controller:
\begin{equation}\label{heat}
\left\{
\begin{array}{ll}
\displaystyle y_t-\Delta y + a\int_0^t y(s)ds   = u\chi_{\omega}(x) &  \mbox{in}  \    Q,  \\[2mm]
y = 0 & \mbox{on} \  \Sigma, \\[2mm]
y(0) = y_0 & \mbox{in} \ \Omega,
\end{array}
\right.
\end{equation}
where $a \in \mathbb{R}$. Clearly, when $\omega=\Omega$, the control $u$ can absorb the memory term ``$ a\int_0^t y(s)ds$", and therefore, one can easily obtain the null controllability of \eqref{heat} for this special case. However, when $\omega$ is a proper subset of $\Omega$, by  \cite{FI, G-I, Ha-Pan, ZG}, the equation \eqref{heat} is null controllable if and only if $a=0$, i. e. in the absence of memory terms. This indicates that \eqref{heat} is not null controllable (needless to say  memory-type null controllable) whenever $a\not=0$ and $\omega\subsetneq\Omega$. Because of this, and inspired by \cite{CRZ, Khapalov, MRR,  RZ2012} (and also \cite{Lions92, LiuY, Z} for the wave equations),  in order to obtain the memory-type null controllability for parabolic equations, we need to make the controller to move so that its support covers the whole domain $\Omega$ during the control time horizon $[0, T]$.

Now, for a given (space-independent) memory kernel $M (\cdot)\in L^1(0,T)$, we consider the following heat equation with  memory, and with  a moving control region $\omega(\cdot)(\subset\Omega)$:
\begin{equation}\label{heat-moving}
\left\{
\begin{array}{ll}
\displaystyle y_t-\Delta y + \int_0^t M(t-s) y(s)ds   = u\chi_{\omega(t)}(x) &  \mbox{in}  \    Q,  \\[2mm]
y = 0 & \mbox{on} \  \Sigma, \\[2mm]
y(0) = y_0 & \mbox{in} \ \Omega.
\end{array}
\right.
\end{equation}
In the next subsection we make precise the assumptions that are required for the moving control support and the consequences this leads to, that will be the key to address the control of this memory heat equation.

\subsection{Preliminaries on moving controls}

In \cite{CRZ}, devoted to the control of the system of viscoelasticity, the authors faced the same difficulty according to which the support of the control needs to move in time and cover the whole domain $\Omega$ to ensure the null-contrabillity of the full system. We recall here the main assumptions on the moving control in \cite{CRZ} and the results it leads to in terms of Carleman inequalities, that will play an essential role when considering the parabolic equation with memory.

We shall consider the control region $\omega(\cdot)$ determined by the evolution of a given reference subset through a flow $X(x,t,t_0)$, which is generated by some vector field $f\in C([0,T];W^{2,\infty}(\R ^n;$ $\R ^n))$, i.e. $X$ solves
$$
\left\{
\begin{array}{l}
\displaystyle \frac{\partial X(x,t,t_0)}{\partial t}= f ( t,X(x,t,t_0)),\ \ t\in [0,T],\\[2mm]
X(x,t_0,t_0)=x\in \R ^n.
\end{array}
\right.
$$
More precisely, we need the following condition (introduced in \cite{CRZ}):

\begin{assumption}\label{tso1}
There exists a flow $X(x,t,0)$ generated by some $f\in C([0,T];W^{2,\infty}(\R ^n;\R ^n))$, a bounded, smooth and  open set $\omega _0\subset \R ^n$, a curve $\Gamma (\cdot)\in C^\infty ([0,T];\R ^n)$,
and two numbers $t_1$ and $t_2$ with $0\le t_1 < t_2 \le T $ such that
 $$
\left\{
\begin{array}{ll}
\Gamma (t) \in X(\omega _0,t,0)\cap \Omega , \quad \forall\; t\in [0,T] , \\[2mm]
\overline{\Omega} \subset \displaystyle\bigcup_{t\in [0,T]} X(\omega _0,t,0) \equiv \{ X(x,t,0)\;|\; x\in \omega _0,\ t\in [0,T]\} , \\[2mm]
\Omega \setminus \overline{X(\omega _0,t,0) } \text{ is nonempty and connected for } t\in [0,t_1]\cup [t_2,T], \\ [2mm]
\Omega \setminus \overline{X(\omega _0,t,0) } \text{ has two (nonempty) connected components  for } t\in (t_1,t_2),\\[2mm]
\forall\; \gamma(\cdot)\in C([0,T]; \Omega ) , \ \exists\; \bar t\in [0,T]\hbox{ satisfying } \gamma (\bar t\,) \in X(\omega _0,\bar t,0).
\end{array}
\right.
 $$
\end{assumption}

We will need to use a known weight function, stated in the following result.
\begin{lemma}\label{weightpsi}
\label{weight}
{\rm (\cite{CRZ})} Let Assumption \ref{tso1} hold, and let $\omega$ and $\omega_1$ be  any two nonempty open sets in $\mathbb{R}^n$ such that
$\overline{\omega _0}\subset \omega _1$ and $ \overline{\omega _1} \subset \omega$.  Then there exist a number $\delta \in (0,T/2)$ and
 a function $\psi \in C^\infty ( \overline{Q})$ such that
$$
\left\{
\begin{array}{ll}
\nabla \psi(t,x) \ne 0,&\quad t\in[0,T],\ x\in \overline{\Omega} \setminus X(\omega _1,t,0),  \\[2mm]\displaystyle
\psi _t (t,x) \ne 0,  &  \quad  t\in[0,T],\ x\in \overline{\Omega}\setminus X(\omega _1,t,0),  \\[2mm]\displaystyle
\psi _t (t,x)  >0,  & \quad  t\in[0,\delta ],\ x\in \overline{\Omega}\setminus X(\omega _1,t,0),  \\[2mm]\displaystyle
\psi _t  (t,x) <0,  &  \quad  t\in[T-\delta ,T ],\ x\in \overline{\Omega}\setminus  X(\omega _1,t,0), \\[2mm]\displaystyle
\frac{\partial \psi}{\partial \nu}(t,x) \le 0,  & \quad  t\in [0,T ],\ x\in \partial \Omega , \\[2mm]\displaystyle
\psi (t,x) >\frac{3}{4}|\psi |_{L^\infty (Q) }, &\quad  t\in [0,T ],\ x\in \overline{ \Omega} .
 \end{array}
\right.
 $$
\end{lemma}

As in \cite{CRZ}, we take a function  $g\in C^\infty (0,T)$ such that
\[
g(t) = \left\{
\begin{array}{ll}
\frac{1}{t} \qquad &\text{for } 0<t<\delta /2 , \\[2mm]
\text{\rm strictly decreasing}\qquad &\text{for }  0<t \le  \delta ,\\[2mm]
1 \qquad &\text{for } \delta \le t \le \frac{T}{2},\\[2mm]
g(T-t)\qquad &\text{for }   \frac{T}{2} \le t < T
\end{array}
\right.
\]
 and define the following two weight functions on $Q$:
 $$
\varphi (t,x) = g(t)\left(e^{\frac{3}{2} \lambda |\psi |_{L^\infty(Q)}} - e^{\lambda \psi (t,x)}\right), \qquad
\theta (t,x) = g(t) e^{ \lambda \psi (t,x)},
 $$
where $\lambda >0$ is a parameter.
For any functions $p\in H^{1,2}(Q)$ and $q\in L^2(Q)$ and parameter $s>0$, we introduce the notation
\begin{equation}\label{p-1}
I_H(p) = \int_Q \Big[  (s\theta )^{-1}  ( |\Delta p |^2  +|p_t|^2 )  + \lambda ^2 s\theta  |\nabla p|^2    +  \lambda ^4 (s\theta) ^3 |p|^2  \Big]e^{-2s\varphi} dxdt
\end{equation}
and
\begin{equation}\label{q-1}
I_O(q) = \lambda ^2  s\int_Q\theta  |  q |^2 e^{-2s\varphi }dxdt.
\end{equation}

In the sequel, we will use $C$ to denote a generic positive constant
which may vary from line to line (unless otherwise stated). The following two results are proved in \cite{CRZ}.

\begin{lemma}
\label{lem1}
Let Assumption \ref{tso1} hold and $\omega_1$ be given in Lemma \ref{weight}. Then, there exist two constants  $\lambda _0>0$ and  $s_0>0$ such that  the following estimate
\begin{equation}
I_H(p) \leq C \left( \int_Q |p_t + \Delta p |^2 e^{-2s\varphi} dxdt + \lambda ^4s^3\int_0^T  \int_{X(\omega _1,t,0) }    \theta ^3 |p|^2 e^{-2s\varphi} dxdt \right),  \label{E99}
\end{equation}
holds for any $\lambda \geq \lambda _0$, $s\geq s_0$ and
$p\in C([0,T];L^2(\Omega ))$ with $p_t + \Delta p\in L^2 (0,T;L^2(\Omega ))$.\end{lemma}
\noindent

\begin{lemma}
\label{lem2}
Let Assumption \ref{tso1} hold and $\omega$ be given in Lemma \ref{weight}. Then, there exist two numbers $\lambda _1\geq \lambda _0$ and  $s_1\geq s_0$ such that the following inequality
\begin{equation}
 I_O(q) \leq C\left( \int_Q |q_t|^2 e^{-2s\varphi} dxdt + \lambda ^2 s^2\int_0^T\int_{X(\omega,t,0)  }  \theta ^2  |q|^2  e^{-2s\varphi} dxdt
\right), \label{E50}
\end{equation}
holds for any $\lambda \geq \lambda _1$, $s\geq s_1$ and
$q\in H^1(0,T;L^2(\Omega ))$.
\end{lemma}

As a consequence of Lemma \ref{lem2}, we have the following result.

\begin{corollary}\label{lemma3}
Under the assumptions in Lemma \ref{lem2}, for any $\lambda \geq \lambda _1$ and $s\geq s_1$, $m\in\N$, and
$q\in H^m(0,T;L^2(\Omega ))$, the following estimate holds
\begin{equation}
 I_O(q) + \sum_{k=1}^{m-1}I_O(\partial^k_tq) \leq C\left( \int_Q |\partial^m_tq|^2 e^{-2s\varphi} dxdt + \int_0^T\int_{ X(\omega,t,0)  } (\lambda s \theta) ^{P(m)} |q|^2  e^{-2s\varphi} dxdt
\right), \label{E50-d}
\end{equation}
where $P(m)$ is polynomial in $m$.
\end{corollary}
\begin{proof}
Assume $m\ge 2$. We consider the equation
$$
\partial^m_tq = f,
$$
which can be rewritten as
\begin{equation}\label{adjointk=kk-2-h}
\left\{
\begin{array}{ll}
\partial_tq^{m-1} = f, \\[2mm]
\partial_tq^{m-2} = q^{m-1},   \\[2mm]
\partial_tq^{m-3} = q^{m-2},   \\[2mm]
\vdots \\[2mm]
\partial_tq^{2} = q^{3}, \\[2mm]
\partial_tq^{1} = q^{2},   \\[2mm]
\partial_tq = q^1.
\end{array}
\right.
\end{equation}
Fix a sequence $\{\omega^k\}_{k=1}^{m-1}$ of nonempty open sets in $\R^n$ such that $\overline{\omega _0}\subset \omega_1$, $\overline{\omega _1}\subset\omega _2, \cdots$, $\overline{\omega _{m-2}}\subset\omega _{m-1}$, $\overline{\omega _{m-1}}\subset\omega$. Then, applying Lemma \ref{lem2} to each equation of \eqref{adjointk=kk-2-h}, we obtain, after absorbing the lower order terms, that
\begin{align}\label{M-1}
I_O(q) + \sum_{k=1}^{m-1}I_O(q^k) & \leq C\biggl( \int_Q |f|^2 e^{-2s\varphi} dxdt + \lambda ^2 s^2\int_0^T\int_{ X(\omega,t,0)  } \theta ^2  |q|^2  e^{-2s\varphi} dxdt  \nonumber \\[2mm]
&\quad+ \lambda ^2 s^2\sum_{k=1}^{m-1}  \int_0^T\int_{ X(\omega _{m-k},t,0)  }\theta ^2  |q^k|^2  e^{-2s\varphi} dxdt   \biggl).
\end{align}

Since $\overline{\omega _1}\subset\omega _2$, we introduce  a cut-off function $\xi\in C_0^\infty(\omega _2; [0,1]) $ such that $\xi=1$ in $\omega _1$ and consider $\zeta(t,x)=\xi (X(x,0,t) )$.
Then, by  \eqref{adjointk=kk-2-h}, it follows
 \begin{equation}\label{zxa1-1}
 \begin{array}{ll}
 \displaystyle\int_0^T\int_{ X(\omega _1,t,0)  } \theta ^2  |q^{m-1}|^2  e^{-2s\varphi} dxdt\\[5mm]
 \displaystyle\leq \int_Q \theta ^2 \zeta |q^{m-1}|^2  e^{-2s\varphi} dxdt
 =\int_Q \theta ^2 \zeta q^{m-1} \partial_tq ^{m-2} e^{-2s\varphi} dxdt.
 \end{array}
 \end{equation}
Moreover, for any $\varepsilon>0$, we have
 \begin{equation}\label{zxa1-2}
 \begin{array}{ll}
 \displaystyle\int_Q \theta ^2 \zeta q^{m-1} \partial_tq^{m-2}  e^{-2s\varphi} dxdt=-\int_Q q^{m-2}\partial_t[ \theta ^2 \zeta q^{m-1} e^{-2s\varphi}] dxdt\\[5mm]
 \ds=-\int_Q q^{m-2}\partial_t q^{m-1} \theta ^2 \zeta  e^{-2s\varphi} dxdt  -\int_Q q^{m-2} q^{m-1} \partial_t[ \theta ^2 \zeta e^{-2s\varphi}] dxdt \\[5mm]
 \displaystyle\leq \varepsilon\left[\frac{1}{\lambda^2s^2}\int_Q |\partial_tq^{m-1}|^2e^{-2s\varphi} dxdt+\frac{1}{s}\int_Q\theta|q^{m-1}|^2e^{-2s\varphi} dxdt\right]\\[5mm]
 \displaystyle\quad+\frac{C\lambda^2s^3}{\varepsilon}\int_0^T\int_{ X(\omega _2,t,0)  } \theta ^7  |q^{m-2}|^2  e^{-2s\varphi} dxdt,
 \end{array}
 \end{equation}
 where we have used the fact that $ |\partial_t(\theta^2 \zeta e^{-2s\varphi})| \leq Cs\zeta \theta^4 e^{-2s\varphi}$.

Hence, by \eqref{zxa1-1}--\eqref{zxa1-2}, \eqref{adjointk=kk-2-h} and \eqref{q-1}, we find that
  \begin{equation}\label{zxa1-3}
 \begin{array}{ll}
 \displaystyle \lambda ^2 s^2\int_0^T\int_{ X(\omega _1,t,0)  } \theta ^2  |q^{m-1}|^2  e^{-2s\varphi} dxdt\\[2mm]
 \displaystyle \leq \varepsilon\left[\int_Q |f|^2 e^{-2s\varphi} dxdt +I_O(q^{m-1})\right]+\frac{C\lambda^4s^5}{\varepsilon}\int_0^T\int_{ X(\omega _2,t,0)  } \theta ^7  |q^{m-2}|^2  e^{-2s\varphi} dxdt.
 \end{array}
 \end{equation}
 Combining \eqref{M-1} and \eqref{zxa1-3}, we end up with
\begin{align}\label{Mzz-1}
I_O(q) + \sum_{k=1}^{m-1}I_O(q^k) & \leq C\biggl( \int_Q |f|^2 e^{-2s\varphi} dxdt + \lambda ^2s^2\int_0^T\int_{ X(\omega,t,0)  } \theta ^2  |q|^2  e^{-2s\varphi} dxdt  \nonumber \\[2mm]
&\quad+ \lambda ^4 s^5\sum_{k=1}^{m-2}  \int_0^T\int_{ X(\omega _{m-k},t,0)  }\theta ^7  |q^k|^2  e^{-2s\varphi} dxdt   \biggl).
\end{align}
Similar to \eqref{zxa1-3}, we have
\begin{equation}\label{zzxa1-3}
 \begin{array}{ll}
 \displaystyle \lambda ^4 s^5\int_0^T\int_{ X(\omega _2,t,0) } \theta ^7  |q^{m-2}|^2  e^{-2s\varphi} dxdt\\[2mm]
 \displaystyle \leq \varepsilon\left[I_O(q^{m-1})+I_O(q^{m-2})\right]+\frac{C\lambda ^8 s^{11}}{\varepsilon}\int_0^T\int_{ X(\omega _3,t,0)  } \theta ^{17}  |q^{m-3}|^2  e^{-2s\varphi} dxdt.
 \end{array}
 \end{equation}
 Combining \eqref{Mzz-1} and \eqref{zzxa1-3}, we conclude that
\begin{align}\label{Mz2z-1}
I_O(q) + \sum_{k=1}^{m-1}I_O(q^k) & \leq C\biggl( \int_Q |f|^2 e^{-2s\varphi} dxdt + \lambda ^2s^2\int_0^T\int_{ X(\omega,t,0)  } \theta ^2  |q|^2  e^{-2s\varphi} dxdt  \nonumber \\[2mm]
&\quad+ \lambda ^8 s^{11}\sum_{k=1}^{m-3}  \int_0^T\int_{ X(\omega _{m-k},t,0)  }\theta ^{17}  |q^k|^2  e^{-2s\varphi} dxdt   \biggl).
\end{align}

Repeating the above argument, we prove \eqref{E50-d}.
\end{proof}

\subsection{Observability}
In order to consider the memory-type null controllability (with the memory kernel $\wt M(\cdot)$) of \eqref{heat-moving}, by \eqref{adj1-z0}, we introduce
the following adjoint system of \eqref{heat-moving}
\begin{equation}\label{heat-adj}
\left\{
\begin{array}{ll}
\displaystyle w_t=-\Delta w - \int_t^T M(s-t) w(s)ds  + \wt M(T-t)z_T &  \mbox{in}  \    Q,  \\[2mm]
w = 0 & \mbox{on} \  \Sigma, \\[2mm]
w(T) = w_T & \mbox{in} \ \Omega,
\end{array}
\right.
\end{equation}
where $w_T,z_T \in L^2(\Omega)$.

In what follows, we choose
 \begin{equation}\label{MM}
 M(t)=e^{at}\sum_{k=0}^K a_kt^k,\quad \wt M(t)=e^{at}\sum_{k=0}^K b_kt^k,
 \ee
where $K\in \mathbb{N}$, and $a, a_0,\cdots,a_K$, $b_0,\cdots,b_K$ are real constants.

We have the following observability result for \eqref{heat-adj}.

\begin{theorem}\label{teo1}
Let Assumption \ref{tso1} hold, $\omega$ be any open set in $\Omega$ such that $\overline{\omega_0} \subset \omega$, and $M$ and $\wt M$ be given by \eqref{MM}.
Then, solutions to \eqref{heat-adj} satisfy
\be \label{o-x}
|w(0)|_{L^2(\Omega)}^2 \leq C\int_0^T\int_{\omega(t)}|w|^2dxdt,\quad \forall\;w_T, z_T \in L^2(\Omega),
\ee
where $\omega(t) :=X(\omega,t,0)$.
\end{theorem}

\begin{proof}
Without loss of generality, we assume that $a=0$ in \eqref{MM} (Otherwise we introduce a function transform $w(\cdot)\to e^{-a\cdot}w(\cdot)$ in \eqref{heat-adj}). Let us write
 \begin{equation}\label{Mz}
 Z = - \int_t^T M(s-t) w(s)ds  + \wt M(T-t)z_T.
 \ee
By \eqref{MM}, we have
\be
\partial^{K+1}_tZ = \sum_{k=0}^K(-1)^kk!a_k\partial_t^{K-k}w.
\ee
Hence, from \eqref{heat-adj}, we see that
\begin{equation}\label{adjointk=kk-1}
\left\{
\begin{array}{ll}
w_{t}+\Delta w = Z &  \mbox{in}  \    Q,  \\[2mm]
\displaystyle \partial^{K+1}_tZ =\sum_{k=0}^Kk!a_k\partial_t^{K-k}w &   \mbox{in}  \    Q,  \\[2mm]
w = 0 & \mbox{on} \  \Sigma.
\end{array}
\right.
\end{equation}

Now,  we take the $K+1$ time derivatives in the first and third equations in \eqref{adjointk=kk-1}. Write $\hat{w} = \partial^{K+1}_tw$. This leads to  the system
\begin{equation}\label{adjointk=kk-2-1}
\left\{
\begin{array}{ll}
\displaystyle\hat w_t+\Delta \hat w =\sum_{k=0}^Kk!a_k\partial_t^{K-k}w  &  \mbox{in}  \    Q,  \\[2mm]
\hat w = 0 & \mbox{on} \  \Sigma, \\[2mm]
\partial^{K+1}_tw =\hat{w} & \mbox{in} \ Q.
\end{array}
\right.
\end{equation}

We apply Lemma  \ref{lem1} and  Corollary \ref{lemma3} to equations in \eqref{adjointk=kk-2-1}. After absorbing the lower order terms, we get
  \begin{align}\label{E99-1-x}
 &I_{H}(\hat w) +I_{O}(w)  + \sum_{i=1}^{K}I_O(\partial^i_t w)    \\[2mm]
 & \leq C \left(  \int_0^T \int_{X(\omega',t,0) }  \lambda ^4  (s\theta )^3 |\hat w|^2 e^{-2s\varphi} dxdt + \int_0^T\int_{X(\omega,t,0)  }  ( \lambda s\theta )^{P(K+1)}  |w|^2
e^{-2s\varphi} dxdt   \right),  \nonumber
\end{align}
where $\omega' $ is a nonempty open subset in $\R^n$ such that $\overline{\omega'} \subset \omega$, and $P(K+1)$ is polynomial in $K$.

Using the last equation in \eqref{adjointk=kk-2-1}, similar to the proof of Corollary \ref{lemma3}, one can show that, for any $\varepsilon>0$,
 \begin{equation}\label{adz-2-1}
 \begin{array}{ll}
\displaystyle\int_0^T \int_{X(\omega',t,0) }  \lambda ^4  (s\theta )^3 |\hat w|^2 e^{-2s\varphi} dxdt \\[3mm]
\displaystyle\leq \varepsilon\left[\int_Q  (s\theta )^{-1} |\partial_t\hat w|^2e^{-2s\varphi} dxdt+I_{O}(w)  + \sum_{i=1}^{K}I_O(\partial^i_t w)\right]\\[3mm]
\displaystyle\quad+\frac{C}{\varepsilon} \int_0^T\int_{X(\omega,t,0)  }  ( \lambda s\theta )^{P(K+1)}  |w|^2
e^{-2s\varphi} dxdt .
\end{array}
\end{equation}

Combining \eqref{E99-1-x} and \eqref{adz-2-1}, we conclude that
 \begin{equation} \label{E99-1-y}
 I_{H}(\hat w) +I_{O}(w)  + \sum_{i=1}^{K}I_O(\partial^i_t w)   \leq C \int_0^T\int_{\omega(t)  }  ( \lambda s\theta )^{P(K)}    | w|^2e^{-2s\varphi} dxdt.
 \end{equation}

From the inequality \eqref{E99-1-y}, and applying the usual energy estimate to the equation \eqref{adjointk=kk-2-1},  we obtain easily the desired estimate \eqref{o-x}. This completes the proof of Theorem \ref{teo1}.
\end{proof}

By Proposition \ref{prop-2}, as a direct consequence of Theorem \ref{teo1}, we have the following memory-type null controllability result for the equation \eqref{heat-moving}:
\begin{theorem}\label{theo-4}
Under the assumptions in Theorem \ref{teo1}, for any $y_0 \in L^2(\Omega)$, there is a control $u\in L^2(Q)$ such that the corresponding solution $y(\cdot)$ to \eqref{heat-moving}, with the control support $\omega(t) =X(\omega,t,0)$ for $t\in(0,T)$, satisfies
$$
y(T) = \int_0^T \wt M(T-s) y(s)ds  = 0 \quad \text{in} \ \Omega.
$$
\end{theorem}

\begin{remark}
To prove Theorem \ref{theo-4}, we need Assumption \ref{tso1}, which seems not to be optimal. For instance, the fourth condition in this assumption, namely, the complement of the control region has two connected components for $t\in( t_1, t_2)$, is necessary for the construction of the basic weight function in Lemma \ref{weightpsi}, and consequently, for the obtainment of the null controllability through the use of Carleman inequalities with moving controls. However, as the result in \cite{MRR} indicates,  in the $1$-d case this condition is not necessary for the null controllability of the structurally damped wave equation with moving control. For a more general discussion on the optimality of Assumption \ref{tso1}, see \cite{CRZ}.
\end{remark}

\begin{remark}
The method developed in this section can also be used to treat the memory-type null controllability problem for the following parabolic equation with memory:
\begin{equation}\label{heat-moving-Laplace}
\left\{
\begin{array}{ll}
\displaystyle y_t-\Delta y + \int_0^t M(t-s)\Delta y(s)ds   = u\chi_{\omega(t)}(x) &  \mbox{in}  \    Q,  \\[2mm]
y = 0 & \mbox{on} \  \Sigma, \\[2mm]
y(0) = y_0 & \mbox{in} \ \Omega,
\end{array}
\right.
\end{equation}
where the kernels $M(\cdot)$ and $\wt M(\cdot)$ are of the form \eqref{MM}. Indeed, in this case, the corresponding adjoint system is given by
\begin{equation}\label{heat-adj-Laplace}
\left\{
\begin{array}{ll}
\displaystyle w_t=-\Delta w - \int_t^T M(s-t) \Delta w(s)ds  + \wt M(T-t)\Delta z_T &  \mbox{in}  \    Q,  \\[2mm]
w = 0 & \mbox{on} \  \Sigma, \\[2mm]
w(T) = w_T & \mbox{in} \ \Omega,
\end{array}
\right.
\end{equation}
where $w_T  \in L^2(\Omega)$ and $z_T \in H^2(\Omega)\cap H^1_0(\Omega)$. Under the assumptions of Theorem \ref{teo1}, we can show that the observability  inequality \eqref{o-x} still holds for the equation \eqref{heat-adj-Laplace}.

Let us show this in the simple case where $M(\cdot) \equiv \wt M(\cdot) \equiv 1$. First, we set
$$
z(t) =  - \int_t^T \Delta w(s)ds  + \Delta z_T,
$$
which leads to the system
\begin{equation}\label{adjoint-Laplace}
\left\{
\begin{array}{ll}
w_{t}+\Delta w = z &  \mbox{in}  \    Q,  \\[2mm]
z_t = \Delta w &   \mbox{in}  \    Q,  \\[2mm]
w = 0 & \mbox{on} \  \Sigma.
\end{array}
\right.
\end{equation}
From \eqref{adjoint-Laplace}, we see that $w$ satisfies
$$
w_{tt} + \Delta w_t - \Delta w=0.
$$
Setting $\theta = w_t -w$, we obtain the system
\begin{equation}\label{adjoint-Laplace-1}
\left\{
\begin{array}{ll}
\theta_{t}+\Delta \theta +\theta  + w =0  &  \mbox{in}  \    Q,  \\[2mm]
w_t -w = \theta &   \mbox{in}  \    Q,  \\[2mm]
\theta = 0 & \mbox{on} \  \Sigma.
\end{array}
\right.
\end{equation}
Arguing as before, we can easily prove a Carleman inequality for the system \eqref{adjoint-Laplace-1}, with an observation in $w$. From such a Carleman inequality, we immediately obtain the observability  inequality \eqref{o-x}.
\end{remark}

\section{Further comments.}\label{s6-1}
The techniques developed in this paper open up the possibility of addressing many other related issues. We mention here some of them that could be of interest for future research:
\begin{itemize}

\item In this work, we have addressed the problem of memory-type controllability  for  finite-dimensional ordinal differential and  parabolic equations. But, even in these cases, we have limited our attention to some very special situations. A systematic analysis of these issues in a broader context is still to be done. In particular, our results for parabolic equations concern mainly the memory kernels of polynomial type. Indeed, from the proof of Theorem \ref{teo1}, it is easy to see that the special form of both $\wt M(\cdot)$ and $M(\cdot)$  in \eqref{MM} plays a crucial role. It would be of interest to extend these results to the more general cases, in the context  of analytic memory kernels.

\item Similar problems could be considered for hyperbolic like equations, for instance, for the wave equation with memory terms. This is an open problem for the general case (See \cite{LZZ1} for some results in the respect).

\item
This work addresses only the memory-type null controllability  problem for linear equations. The same problem would be of interest for nonlinear equations but the methods of proof used in this paper, that allow dealing with special memory kernels, and that require to compute successive time derivates of the system under consideration, do not apply in the nonlinear context.

For instance, it would be quite interesting to consider the memory-type null controllability of the following nonlinear version of \eqref{heat-moving}:
 \begin{equation}\label{non-heat-moving}
\left\{
\begin{array}{ll}
\displaystyle y_t-\Delta y +f(y)+ \int_0^t M(t-s) y(s)ds   = u\chi_{\omega(t)}(x) &  \mbox{in}  \    Q,  \\[2mm]
y = 0 & \mbox{on} \  \Sigma, \\[2mm]
y(0) = y_0 & \mbox{in} \ \Omega.
\end{array}
\right.
\end{equation}
Most often, the controllability of semilinear systems is achieved by a fixed point method, out of the controllability of the linearized one, replacing the nonlinear term with a linear one, involving a $(t,x)$-dependent potential.
The approach developed to derive the observability estimate for \eqref{heat-adj} does not apply to this case and, consequently, the memory-type null controllability problem for \eqref{non-heat-moving} is completely open.

\item
It is also quite interesting to address the  memory-type null controllability for other PDE models. On the other hand, it is also interesting to consider the controllability problems for PDEs involving other types of  nonlocal terms. This issue seems to be widely open. We refer to \cite{UB} for the analysis of nonlocal fractional Schr\"odinger and wave equations from the controllability viewpoint.

\item
One can also introduce the concept of memory-type null controllability for stochastic evolutions. However, generally speaking, one needs to consider memory terms both in drift and diffusion components. Because of this, besides \eqref{0e1s}, one also requires that
 $$
 \int_0^T\h M(T-s)y(s)dW(s)=0,
 $$
where $\h M(\cdot)$ is another memory kernel, and $\{W(t)\}_{t\in[ 0,T]}$ is a standard Brownian motion. Clearly, in general, the corresponding problem is quite challenging even for the  stochastic evolutions in finite dimensions.

\item The memory-type controllability property considered along the paper is a particular instance of more general controllability concepts. For instance,  let $Z$ be a metric space, and $\Gamma$ be a nonempty subset of $Z$. Suppose $F:\ C([0,T]; Y)\times L^2(0,T;U)\to Z$ is a given map. Motivated by \cite[p. 14]{YL}, the equation \eqref{0.2} is said to be $(F,\Gamma)$-controllable if for any $y_0\in Y$, there is a control $u(\cdot)\in L^2(0,T;U)$ such that the corresponding solution $y(\cdot)$ satisfies
 \begin{equation}\label{0e1}
 F(y(\cdot), u(\cdot))\in \Gamma.
 \end{equation}
Obviously, the memory-type null controllability property of \eqref{0.2} is a special case of $(F,\Gamma)$-controllability of the same system with
 \begin{equation}\label{0zxn2}
   Z=Y\times Y, \quad \Gamma=\{0\}\times \{0\}
 \end{equation}
 and
 $$
 F(y(\cdot), u(\cdot))=\left(\begin{array}{cc}\ds y(T)\\[2mm]\ds \int_0^T\wt M(T-s)y(s)ds\end{array}
 \right).
 $$
Nevertheless, so far the $(F,\Gamma)$-controllability concept is too general to obtain meaningful results.

\end{itemize}

\bigskip
\noindent{\bf Acknowledgements:} F. W. Chaves-Silva was partially supported by the ERC project Semi-Classical Analysis of  Partial Differential Equations, ERC-2012-ADG, project number: 320845. X. Zhang was supported by the NSF of China under grant 11231007 and the Chang Jiang Scholars Program from the Chinese
Education Ministry. This work was partially supported by the Advanced Grant DYCON (Dynamic Control) of the European Research Council Executive Agency, ICON of the French ANR (ANR-2016-ACHN-0014-01), FA9550-15-1-0027 of AFOSR, A9550-14-1-0214 of the EOARD-AFOSR, and the MTM2014-52347 Grant of the MINECO (Spain).


\begin{thebibliography}{99}

\bibitem{AI} S.A. Avdonin and S.A. Ivanov. \sl Families of Exponentials. The Method of Moments in Controllability Problems for Distributed Parameter Systems. \rm Cambridge University Press, Cambridge, UK, 1995.

\bibitem{Barbu}  V. Barbu and M. Iannelli. \it Controllability of the heat equation with memory. \sl  Differential Integral Equations. \rm 13 (2000), 1393--1412.

\bibitem{UB} U. Biccari. \it Internal control for evolution equations involving the fractional Laplace operator, arXiv:1411.7800v1.


\bibitem{Castro} C. Castro. \it Exact controllability of the $1-d$ wave equation from a moving interior point. \sl ESAIM Control Optim. Calc. Var. \rm 19 (2013), 301--316.

\bibitem{Ca} C.~Cattaneo. \it  A form of heat conduction equation which eliminates the paradox of instantaneous propagation. \sl  Compute. Rendus. \rm 247 (1958), 431--433.

\bibitem{CRZ}F.~W.~Chaves-Silva, L.~Rosier and E.~Zuazua. \it Null controllability of a system of viscoelasticity with a moving control. \sl J. Math. Pures Appl. \rm 101 (2014), 198--222.

\bibitem{Coron}J.-M. Coron. \sl Control and Nonlinearity. \rm Mathematical Surveys and Monographs, vol.~136. American Mathematical Society, Providence, RI, 2007.

\bibitem{FYZ} X.~Fu, J.~Yong and X.~Zhang. \it Controllability and observability of the heat equations with hyperbolic memory kernel. \sl J. Differential Equations. \rm  247 (2009), 2395--2439.

\bibitem{FI}A.V. Fursikov and O.Yu. Imanuvilov. \sl Controllability of Evolution Equations. \rm Lecture Notes Series, vol.~34. Research Institute of Mathematics, Seoul National University, Seoul, Korea, 1996.


\bibitem{GrosMill} S. ~I. Grossman and R. ~K. Miller. \it Perturbation theory for Volterra integro-differential systems. \sl J. Differential Equations. \rm  8 (1970), 457--474.

\bibitem{G-I}  S.~Guerrero and O.~Yu.~Imanuvilov. \it Remarks on non controllability of the heat equation with memory. \sl ESAIM Control Optim. Calc. Var. \rm 19 (2013), 288--300.

\bibitem{GP} M.~E.~Gurtin and B.~C.~Pipkin. \it A general theory of heat conduction with finite wave speeds. \sl Arch. Rat. Mech. Anal. \rm 31 (1968), 113--126.

\bibitem{Ha-Pan} A.  Halanay and L. Pandolfi. \it  Lack of controllability of the heat equation with memory. \sl  Systems Control Lett. \rm 61 (2012), 999--1002.

\bibitem{15}R.E. Kalman. \it On the general theory of control systems. \rm In: \sl Proc. 1st IFAC Congress, Moscow, 1960, vol.~1. \rm Butterworth,
London, 1961, 481--492.

\bibitem{Khapalov}A. Khapalov. \it
Mobile point controls versus locally distributed ones for the controllability of the semilinear parabolic equation. \sl
SIAM J. Control Optim. \rm  40  (2001), 231--252.

\bibitem{Kim} J. U. Kim. \it Control of a second-order integro-differential equation. \sl SIAM J. Control Optim. \rm 31 (1993), 101--110.

\bibitem{LLTT}J. Le Rousseau, G. Lebeau, P. Terpolilli and E. Tr\'elat. \it Some new results for the controllability of waves equation, \rm preprint.

\bibitem{Lions} J.-L.~Lions. \sl Contr\^olabilit\'e Exacte, Perturbations et
Stabilisation de Syst\`emes Distribu\'es, Tome 1. \rm Recherches en
Math\'ematiques Appliqu\'ees, vol.~8. Masson, Paris, 1988.

\bibitem{Lions92} J.-L. Lions. \it Pointwise control for distributed systems. \rm In: H.T. Banks (Ed.), \sl Control and Estimation in Distributed Parameter Systems. \rm Society for Industrial and Applied Mathematics, Philadelphia, PA, 1992, 1--39.

\bibitem{LiuY} K. Liu and J. Yong. \it  Rapid exact controllability of the wave equation by controls distributed on a
time-variant subdomain. \sl Chin. Ann. Math. Ser. B. \rm 20 (1999), 65--76.

\bibitem{LZZ1}Q. L\"u, X.~Zhang and E.~Zuazua. Qi L ̈u, Xu Zhang, Enrique Zuazua, Null Controllability for Wave Equations with Memory,
2016. $<hal-01408112>$

\bibitem{MRR} P. Martin, L. Rosier and P. Rouchon. \it  Null controllability of the structurally damped wave equation with moving control. \sl
SIAM J. Control Optim. \rm 51 (2013),  660--684.

\bibitem{Pandolfi} L. Pandolfi. \it  Boundary controllability and source reconstruction in a viscoelastic string under external traction. \sl J. Math. Anal. Appl. \rm 407 (2013), 464--479.

\bibitem{RZ2012} L. Rosier and B.-Y. Zhang. \it  Unique continuation property and control for the Benjamin-Bona-Mahony equation on a periodic domain. \sl
J. Differential Equations. \rm 254 (2013), 141--178.

\bibitem{Russell} D.L. Russell. \it Controllability and stabilizability theory
for linear partial differential equations: recent progress and open
problems. \sl  SIAM Rev. \rm 20 (1978), 639--739.

\bibitem{YZ}  M.~Yamamoto and X.~Zhang. \it Global uniqueness and
stability for a class of multidimensional inverse hyperbolic
problems with two unknowns. \sl Appl. Math. Optim. \rm 48 (2003),
211--228.

\bibitem{YL} J.~Yong  and   H.~Lou. \sl A Concise Course on  Optimal Control Theory. \rm Higher Education Press, Beijing, 2006. (In Chinese)

\bibitem{YZ1} J.~Yong  and   X.~Zhang.  \it Heat
equation    with    memory    in    anisotropic
and non-homogeneous media. \sl Acta Math. Sin. Engl. Ser. \rm 27 (2011), 219--254.

\bibitem{Z}X. Zhang. \it Rapid exact controllability of the semilinear wave equation. \sl Chin. Ann. Math. Ser. B. \rm  20 (1999), 377--384.

\bibitem{Z2} X.~Zhang. \it A unified controllability/observability theory for some stochastic
and deterministic partial differential equations. \rm In: \sl  Proceedings
of the International Congress of Mathematicians, Vol. IV. \rm
Hyderabad, India, 2010, 3008--3034.

\bibitem{ZG} X.~Zhou and H. Gao. \it Interior approximate and null controllability of the heat equation with memory. \sl Comput. Math. Appl. \rm 67 (2014), 602--613.

\bibitem{Zua2}E. Zuazua. \it Controllability and observability of partial
differential equations: some results and open problems. \rm In: \sl Handbook of
Differential Equations: Evolutionary Differential Equations, vol.~3. \rm
Elsevier Science, 2006, 527--621.

\end{thebibliography}
\end{document}